\documentclass[11pt,reqno]{amsproc}
\usepackage{amsmath}
\usepackage{amssymb}
\usepackage[colorlinks=true]{hyperref}
\hypersetup{urlcolor=red, citecolor=blue}

\usepackage{etoolbox}
\makeatletter
\patchcmd{\@maketitle}{\global\topskip8pc\relax}{\global\topskip4pc\relax}{}{}
\makeatother

\theoremstyle{plain}
\newtheorem*{mainthm}{Theorem}
\newtheorem*{maincor}{Corollary}
\newtheorem{theorem}{Theorem}[section]

\newtheorem{lemma}[theorem]{Lemma}

\theoremstyle{definition}

\numberwithin{equation}{section}

\def\N{\mathbb N}
\def\R{\mathbb R}
\def\Q{\mathbb Q}
\def\m{\mathcal M}
\def\e{\mathcal E}
\def\phi{\varphi}
\def\epsilon{\varepsilon}
\def\LC{LC}

\begin{document}

\title[Convergence of equilibrium measures at temperature zero]{An elementary proof of Br\'emont's theorem on the convergence of
equilibrium measures at temperature zero}

\author{Oliver Jenkinson and Mark Pollicott}

\address{Oliver Jenkinson; School of Mathematical Sciences, Queen
  Mary University of London, Mile End Road, London, E1 4NS, UK.
  \newline {\tt o.jenkinson@qmul.ac.uk}}

\address{Mark Pollicott;
Mathematics Institute, University of Warwick, Coventry, CV4 7AL, UK.
\newline
  {\tt m.pollicott@warwick.ac.uk}}

\begin{abstract}
We give an elementary proof that if $f$ is a locally constant function on
a topologically mixing subshift of finite type, then the equilibrium
measures for $tf$ converge as $t\to\infty$, a result due to Br\'emont.
The proof shows that the limit measure assigns an
algebraic number to every cylinder.
\end{abstract}

\maketitle

\vspace{-2em}

\section*{Introduction}

Let $\sigma:X\to X$ be a topologically mixing subshift of finite type,
and let $f:X\to\R$ be continuous.
For $t\in\R$, an equilibrium measure for $tf$ is a $\sigma$-invariant
Borel probability measure $m$ maximizing $h(m)+t\int f\,dm$, where $h(m)$
denotes entropy; if $f$ is sufficiently regular (e.g.~locally constant)
then for all $t$ there is a unique such measure, denoted $m_{tf}$.
In the thermodynamic interpretation (cf.~\cite{ruellebook}) the parameter
$t$ is an inverse temperature, so the behaviour of $m_{tf}$ as $t\to\infty$
is a zero temperature limit, and an accumulation point of the $m_{tf}$ can
be interpreted as a ground state.
It is known (see Lemma~\ref{lem:extremal}) that all such accumulation
points maximize $\int f\,dm$, and have maximal entropy among the measures
with this property; in particular, if only one maximizing measure has
maximal entropy then $m_{tf}$ converges to it as $t\to\infty$.
In general the limit need not exist: Chazottes \& Hochman
\cite{chazotteshochman} constructed Lipschitz functions $f$ for which
$m_{tf}$ does not converge, and Coronel \& Rivera-Letelier
\cite{coronelriveraletelier} showed that the low temperature behaviour of
$m_{tf}$ can depend sensitively on $f$ (see also \cite{vanenterruszel}).

If $f$ is locally constant then $m_{tf}$ does converge as $t\to\infty$.
This was proved by Br\'emont \cite{bremont},
 who observed that the measure
of each cylinder is a semialgebraic function of the entries of the
transfer matrix of $tf$, which are either $0$ or of the form $e^{\gamma t}$
with $\gamma$ a value of $f$, and deduced
convergence from a general result on subanalytic functions evaluated
along curves of the form $x\mapsto(x^{\epsilon_1},\ldots,x^{\epsilon_n})$
as $x\to0$; this result was proved using the preparation theorem for
subanalytic functions (see \cite{lionrolin,parusinski}).
His approach also shows that, for functions $f$ depending on a fixed
number of coordinates, only finitely many limit measures can arise.
A dynamical proof was given by
Leplaideur \cite{leplaideur}, an algorithmic proof based on
renormalization by Chazottes, Gambaudo \& Ugalde \cite{cgu}, and an
algorithm of Puiseux type describing the limit, for transition graphs
with few vertices, by Garibaldi \& Thieullen \cite{garibaldithieullen}.
The difficulty lies in the case where the maximizing measures of maximal
entropy are not unique: the limit is then a convex combination of the
ergodic ones, and the particular convex combination is not determined by
the set of maximizing measures alone.

In this note we give a short and elementary proof of Br\'emont's theorem.
Our starting point is a standard formula (Lemma~\ref{lem:formula}),
obtained by implicit differentiation of the reciprocal of a dynamical
zeta function, which expresses $m_{tf}(g)$, for locally constant $g$, as a
rational function of $e^{-P(tf)}$ whose coefficients are finite linear
combinations of exponentials $e^{\gamma t}$ with $\gamma\in\R$; here $P$
denotes topological pressure.
A direct asymptotic analysis of this formula is complicated by
cancellation: its numerator and denominator both tend to zero whenever
there is more than one maximizing measure of maximal entropy.
Instead we note that $e^{-P(tf)}$ is itself a zero of a polynomial whose
coefficients are exponential sums of the same kind.
The key point is that eliminating $e^{-P(tf)}$ from these two relations
shows that $m_{tf}(g)$ is also a zero of such a polynomial, with
coefficients not all identically zero; in other words, $m_{tf}(g)$ is
algebraic over the field generated by the functions $t\mapsto e^{\gamma t}$.
Br\'emont's theorem then follows from the elementary observation
(Lemma~\ref{lem:bounded}) that a bounded continuous function on $\R$ with
this property has a limit at infinity.
This observation takes the place of the preparation theorem in
Br\'emont's approach; alternatively, since $t\mapsto m_{tf}(g)$ is
definable in the real exponential field, 
Wilkie's theorem that this
structure is o-minimal \cite{wilkie} implies that $t\mapsto m_{tf}(g)$ is
eventually monotone, and therefore convergent since it is bounded.
The proof also shows that the limit measure assigns an algebraic
number to every cylinder, even when $f$ takes transcendental values, with an
explicit bound on its degree (see the Corollary in Section~\ref{sec:proof}).

\subsection*{Acknowledgments}
We are grateful to Thierry Bousch for helpful comments on a preliminary version of this article.

\section{Preliminaries and statement of result}\label{sec:prelim}

In this section we fix notation, recall some known facts, and state
Br\'emont's theorem.

Let $A$ be a $k\times k$ matrix with entries in $\{0,1\}$, and let $X$ be
the set of sequences $x=(x_n)_{n=1}^\infty$ such that
$x_n\in\{1,\ldots,k\}$ and $A(x_n,x_{n+1})=1$ for all $n\ge1$.
We assume that $X$ is non-empty, and equip it with the product topology,
so that $X$ is compact and the shift map $\sigma:X\to X$, defined by
$(\sigma x)_n:=x_{n+1}$, is continuous.
The dynamical system $(X,\sigma)$ is a \emph{subshift of finite type}, and
we assume throughout that it is \emph{topologically mixing}, i.e.~there
exists $n\ge1$ such that $A^n(i,j)>0$ for all $1\le i,j\le k$.
A \emph{cylinder} of length $n$ is a non-empty set of the form
$[w_1\cdots w_n]:=\{x\in X: x_i=w_i \text{ for } 1\le i\le n\}$.
Let $C(X)$ denote the space of continuous functions $X\to\R$, equipped
with the uniform norm $\|\cdot\|_\infty$.
A function $\phi:X\to\R$ is \emph{locally constant} if there exists
$n\ge1$ such that $\phi$ is constant on all cylinders of length $n$;
the set $\LC(X)$ of locally constant functions is a dense subspace of
$C(X)$.

Let $\m$ denote the set of $\sigma$-invariant Borel probability measures
on $X$, which is compact in the weak-$*$ topology, and let $h(m)$ denote
the entropy of $m\in\m$.
For $\phi\in C(X)$, its \emph{pressure} is
$P(\phi):=\sup\{h(m)+m(\phi):m\in\m\}$, where $m(\phi):=\int\phi\,dm$,
and any $m\in\m$ satisfying
$P(\phi)=h(m)+m(\phi)$ is called an \emph{equilibrium measure} for $\phi$.
The entropy map $m\mapsto h(m)$ is upper semi-continuous on $\m$
(see \cite{walters}), so all $\phi\in C(X)$ have at least one equilibrium
measure, and all $\phi\in\LC(X)$ have a unique equilibrium measure
\cite[Cor.~3.3 (i)]{walterstams}, denoted $m_\phi$.

For $f\in C(X)$ let $\beta(f):=\max\{m(f):m\in\m\}$.
Any $m\in\m$ satisfying $m(f)=\beta(f)$ is called an
\emph{$f$-maximizing measure} (see \cite{bochi,jsurvey1,jsurvey2} for background).
The set $\m_{\max}(f)$ of $f$-maximizing measures is a non-empty compact
convex subset of $\m$, and we set
$h_{\max}(f):=\max\{h(m):m\in\m_{\max}(f)\}$.
Any $m\in\m_{\max}(f)$ with $h(m)=h_{\max}(f)$ is called an
\emph{$f$-extremal measure} (cf.~\cite{coelho,jetds,pollicottsharp}).
The following known result describes the possible accumulation points of
$m_{tf}$ as $t\to\infty$.

\begin{lemma}\label{lem:extremal}
Let $f\in\LC(X)$.
As $t\to\infty$, all weak-$*$ accumulation points of $m_{tf}$ are
$f$-extremal measures, and $P(tf)-t\beta(f)$ decreases to the limit
$h_{\max}(f)$.
\end{lemma}
\begin{proof}
The result seems to have first appeared in the unpublished manuscript
\cite{conzeguivarch}; see \cite[Thm.~1]{morris} for a proof in a more
general (countable alphabet) setting.
The first part of the lemma appears in \cite{coelho}, and published
versions appear in \cite[Prop.~29 (i)]{clt} and \cite[Lem.~14]{jetds}.
\end{proof}

In particular, if $f$ has a unique extremal measure then
$\lim_{t\to\infty}m_{tf}$ exists.
If $f\in\LC(X)$ then there is a subshift of finite type $Y\subseteq X$
such that a measure in $\m$ is $f$-maximizing if and only if its support
is contained in $Y$ (see \cite[Prop.~10]{bouschwalters},
\cite[Thm.~7]{jtams}, \cite{ziemian}).
The $f$-extremal measures are then the measures of maximal entropy for
$\sigma|_Y$, and since an irreducible subshift of finite type has a unique
measure of maximal entropy (see e.g.~\cite{walters}), $f$ has a unique
extremal measure if and only if exactly one irreducible component of $Y$
has topological entropy $h_{\max}(f)$.
If at least two components have this entropy, Lemma~\ref{lem:extremal}
gives only that all accumulation points of $m_{tf}$ are convex
combinations of their measures of maximal entropy.

The following result is due to Br\'emont \cite{bremont}; we give a new
proof in Sections~\ref{sec:formula} and~\ref{sec:proof}.

\begin{mainthm}
Let $(X,\sigma)$ be a topologically mixing subshift of finite type, let
$f:X\to\R$ be locally constant, and for $t\in\R$ let $m_{tf}$ denote the
unique equilibrium measure for $tf$.
Then $m_{tf}$ converges in the weak-$*$ topology as $t\to\infty$.
\end{mainthm}

\section{A formula for $m_{tf}(g)$}\label{sec:formula}

In this section we fix $f,g\in\LC(X)$, and express $m_{tf}(g)$ in terms of
a zero of a polynomial whose coefficients are finite linear combinations
of exponentials in $t$.
Let $r\ge2$ be such that $f$ and $g$ are constant on all cylinders of
length $r$.
By passing to the higher block presentation of $X$ by words of length
$r-1$ (see e.g.~\cite[p.~12]{lindmarcus}) we may, and do, assume that $X$
is the subshift of finite type defined by a primitive $N\times N$ matrix
$B$ with entries in $\{0,1\}$, where $N$ is the number of cylinders of
length $r-1$ in the original subshift, and that $f$ and $g$ depend only on
the first two coordinates; we write $f(i,j)$ and $g(i,j)$ for their values
on the cylinder $[ij]$.
For $t,u\in\R$, define the $N\times N$ matrix $B_{t,u}$ by
$B_{t,u}(i,j):=e^{tf(i,j)+ug(i,j)}$ if $B(i,j)=1$, and
$B_{t,u}(i,j):=0$ if $B(i,j)=0$, and set
\begin{equation*}
\Delta(z,t,u):=\det(I-zB_{t,u}) \quad\text{for } z\in\R\,.
\end{equation*}
The matrix $B_{t,u}$ is primitive, and its spectral radius, which is a
simple eigenvalue, is equal to $e^{P(tf+ug)}$ (see
\cite{parrypollicott,ruellebook}).
We note that
$1/\Delta(z,t,u)=\exp\sum_{n\ge1}\frac{z^n}{n}\sum_{\sigma^nx=x}e^{S_n(tf+ug)(x)}$
for $|z|<e^{-P(tf+ug)}$, where $S_n\phi:=\sum_{i=0}^{n-1}\phi\circ\sigma^i$,
so that $1/\Delta(\cdot,t,u)$ is the dynamical zeta function of $tf+ug$
(see e.g.~\cite[pp.~27--28]{baladi}).

\begin{lemma}\label{lem:formula}
For all $t\in\R$, the number $z(t):=e^{-P(tf)}$ is a simple zero of
$\Delta(\cdot,t,0)$, and
\begin{equation}\label{eq:formula}
m_{tf}(g)\,z(t)\,\partial_z\Delta(z(t),t,0)=\partial_u\Delta(z(t),t,0)\,.
\end{equation}
Moreover, the functions $t\mapsto z(t)$ and $t\mapsto m_{tf}(g)$ are real
analytic on $\R$.
\end{lemma}
\begin{proof}
For $t,u\in\R$, set $z(t,u):=e^{-P(tf+ug)}$.
We have $\Delta(z,t,u)=\prod_{i=1}^N(1-z\lambda_i)$, where
$\lambda_1,\ldots,\lambda_N$ are the eigenvalues of $B_{t,u}$, repeated
according to algebraic multiplicity.
Since $e^{P(tf+ug)}$ is a simple eigenvalue, it follows that $z(t,u)$ is
a simple zero of $\Delta(\cdot,t,u)$.
The coefficients of the polynomial $\Delta(\cdot,t,u)$ are real analytic
functions of $(t,u)$, and $z(t,u)$ depends continuously on $(t,u)$, so
$(t,u)\mapsto z(t,u)$ is real analytic by the implicit function theorem.
Differentiating the identity $\Delta(z(t,u),t,u)=0$ with respect to $u$
gives
$\partial_z\Delta(z(t,u),t,u)\,\partial_u z(t,u)+\partial_u\Delta(z(t,u),t,u)=0$.
Now $\partial_uP(tf+ug)|_{u=0}=m_{tf}(g)$ (see e.g.~\cite[p.~133]{ruellebook}),
so $\partial_uz(t,0)=-m_{tf}(g)\,z(t)$, and (\ref{eq:formula}) follows.
In particular, $m_{tf}(g)=-\partial_uz(t,0)/z(t)$ is real analytic in $t$.
\end{proof}

The coefficients of $\Delta$ have an explicit form.
Let $\e$ denote the set of functions $q:\R\to\R$ of the form
$q(t)=\sum_{k=1}^K b_ke^{\gamma_kt}$, where $K\in\N$ and
$b_k,\gamma_k\in\R$ for $1\le k\le K$, and let $\e_\Q$ denote the set of
those $q\in\e$ admitting such a representation with $b_1,\ldots,b_K\in\Q$;
clearly $\e$ and $\e_\Q$ are closed under addition and multiplication.
Expanding the determinant as a sum over subsets $J\subseteq\{1,\ldots,N\}$
and permutations $\pi$ of $J$ gives
$\Delta(z,t,u)=1+\sum_{n=1}^Nc_n(t,u)\,z^n$, where each $c_n(t,u)$ is a
finite sum of terms $\pm e^{tf_{J,\pi}+ug_{J,\pi}}$, with
$f_{J,\pi}:=\sum_{i\in J}f(i,\pi(i))$ and
$g_{J,\pi}:=\sum_{i\in J}g(i,\pi(i))$, for some $J$ of cardinality $n$ and
some permutation $\pi$ of $J$ such that $B(i,\pi(i))=1$ for all $i\in J$.
In particular, for $1\le n\le N$ we have
\begin{gather}
c_n(\cdot,0)\in\e \quad\text{and}\quad \partial_uc_n(\cdot,0)\in\e\,,
\label{eq:coeffs}\\
c_n(\cdot,0)\in\e_\Q \quad\text{and}\quad \partial_uc_n(\cdot,0)\in\e_\Q
\quad\text{if $g$ takes values in $\Q$}\,.
\label{eq:coeffsQ}
\end{gather}

\section{Proof of the Theorem}\label{sec:proof}

We now complete the proof of the Theorem, using the fact that
$e^{-P(tf)}$ is a zero of the polynomial $\Delta(\cdot,t,0)$, whose
coefficients belong to $\e$ by (\ref{eq:coeffs}).
The only further ingredient needed is the following elementary lemma.

\begin{lemma}\label{lem:bounded}
Let $y:\R\to\R$ be continuous and bounded.
Suppose there exist $D\in\N$ and $q_0,\ldots,q_D\in\e$, not all
identically zero, such that
\begin{equation}\label{eq:algrel}
\sum_{j=0}^D q_j(t)\,y(t)^j=0 \quad\text{for all } t\in\R\,.
\end{equation}
Then $\lim_{t\to\infty}y(t)$ exists.
If, moreover, $q_0,\ldots,q_D\in\e_\Q$, then $\lim_{t\to\infty}y(t)$ is an
algebraic number of degree at most $D$.
\end{lemma}
\begin{proof}
Collecting terms with equal exponents, and discarding those exponents for
which all coefficients vanish, we may write
\begin{equation*}
q_j(t)=\sum_{k=1}^K b_{j,k}\,e^{\gamma_kt}\qquad\text{for } 0\le j\le D\,,
\end{equation*}
where $\gamma_1>\gamma_2>\cdots>\gamma_K$ and, for all $1\le k\le K$, the
coefficients $b_{0,k},\ldots,b_{D,k}$ are not all zero; here $K\ge1$,
since the $q_j$ are not all identically zero.
The polynomial
\begin{equation*}
Q(\xi):=\sum_{j=0}^D b_{j,1}\,\xi^j
\end{equation*}
is therefore non-zero, and $e^{-\gamma_1t}q_j(t)\to b_{j,1}$ as
$t\to\infty$, for $0\le j\le D$.
If $t_i\to\infty$ and $y(t_i)$ converges to some $y^*\in\R$, then multiplying
(\ref{eq:algrel}) by $e^{-\gamma_1t}$, setting $t=t_i$, and letting
$i\to\infty$, gives
\begin{equation*}
0=\lim_{i\to\infty}\sum_{j=0}^D e^{-\gamma_1t_i}q_j(t_i)\,y(t_i)^j
=\sum_{j=0}^D b_{j,1}\,(y^*)^j=Q(y^*)\,.
\end{equation*}
Since $y$ is bounded, at least one such $y^*$ exists, so $Q$ is not
constant, and its set of zeros is finite.
Set $\underline y:=\liminf_{t\to\infty}y(t)$ and
$\overline y:=\limsup_{t\to\infty}y(t)$.
By continuity of $y$, all points of the interval
$[\underline y,\overline y]$ are of the form $\lim_{i\to\infty}y(t_i)$
for some sequence $t_i\to\infty$, so this interval is contained in the
finite set of zeros of $Q$.
Therefore $\underline y=\overline y$, so $\lim_{t\to\infty}y(t)$ exists,
and it is a zero of $Q$.
If $q_0,\ldots,q_D\in\e_\Q$ then the above representations can be obtained
by collecting terms in representations with rational coefficients, so the
$b_{j,k}$ are rational.
In this case $Q$ is a non-zero polynomial of degree at most $D$ with
rational coefficients, so its zero $\lim_{t\to\infty}y(t)$ is an algebraic
number of degree at most $D$.
\end{proof}

We are now able to prove the Theorem.

\begin{proof}[Proof of Theorem]
Since $\LC(X)$ is dense in $C(X)$, it suffices to show that the limit of
$m_{tf}(g)$ as $t\to\infty$ exists for all $g\in\LC(X)$.
Fix $g\in\LC(X)$, let $B$, $N$, $\Delta$ and $c_n$ be as in
Section~\ref{sec:formula}, and for $t\in\R$ set
\begin{equation*}
y(t):=m_{tf}(g) \quad\text{and}\quad z(t):=e^{-P(tf)}\,.
\end{equation*}
By Lemma~\ref{lem:formula}, $y$ and $z$ are real analytic, and
$|y(t)|\le\|g\|_\infty$ for all $t\in\R$.

The rest of the proof uses only elementary algebra.
Let $\mathcal A$ denote the ring of real analytic functions $\R\to\R$.
If $a,b\in\mathcal A$ satisfy $ab=0$ and $a\ne0$, then $b$ vanishes outside
the zero set of $a$, which is discrete, so $b=0$ by continuity; in other
words, $\mathcal A$ is an integral domain.
Consequently (see \cite[\S7.5, Thm.~15]{dummitfoote})
$\mathcal A$ is contained in a field $\mathcal K$, its field of fractions,
whose elements are the quotients $a/b$ with $a,b\in\mathcal A$ and $b\ne0$,
added and multiplied according to the usual rules for fractions, and with
$a/b=a'/b'$ if and only if $ab'=a'b$.
(Thus $a/b$ may be regarded as the function $t\mapsto a(t)/b(t)$, defined
outside the discrete zero set of $b$.)
Since $\e$ is closed under addition and multiplication, the quotients
$p/q$ with $p,q\in\e$ and $q\ne0$ form a subfield $F$ of $\mathcal K$; in
particular, $\mathcal K$ is a vector space over $F$.

By (\ref{eq:coeffs}), the formulae
\begin{equation*}
\Phi(\xi):=1+\sum_{n=1}^Nc_n(\cdot,0)\,\xi^n\,,\quad
\Psi_1(\xi):=\sum_{n=1}^Nn\,c_n(\cdot,0)\,\xi^n\,,\quad
\Psi_2(\xi):=\sum_{n=1}^N\partial_uc_n(\cdot,0)\,\xi^n
\end{equation*}
define polynomials with coefficients in $\e$, and hence in $F$.
Evaluating them at $\xi=z$ gives the elements of $\mathcal A$ defined, for
$t\in\R$, by
\begin{equation*}
\Phi(z)(t)=\Delta(z(t),t,0)\,,\quad
\Psi_1(z)(t)=z(t)\,\partial_z\Delta(z(t),t,0)\,,\quad
\Psi_2(z)(t)=\partial_u\Delta(z(t),t,0)\,,
\end{equation*}
so Lemma~\ref{lem:formula} gives
\begin{equation}\label{eq:zy}
\Phi(z)=0 \quad\text{and}\quad y\,\Psi_1(z)=\Psi_2(z)\,,
\end{equation}
where $\Psi_1(z)$ has no zeros, since $z(t)$ is a positive simple zero of
$\Delta(\cdot,t,0)$ for all $t\in\R$.

We now show that the $N+1$ elements $1,y,\ldots,y^N$ lie in a vector
space over $F$ of dimension at most $N$, and are therefore linearly
dependent over $F$.
Let $d$ be the degree of $\Phi$, i.e.~the largest $n$ such that
$c_n(\cdot,0)$ is not identically zero; such $n$ exists, and $1\le d\le N$,
because $\Phi(z)=0$ while the constant term of $\Phi$ is $1$.
Let $V$ denote the vector space over $F$ spanned by $1,z,\ldots,z^{d-1}$,
i.e.~the set of elements $p_0+p_1z+\cdots+p_{d-1}z^{d-1}$ of $\mathcal K$
with $p_0,\ldots,p_{d-1}\in F$; its dimension is at most $d$, and hence at
most $N$.
Since $c_d(\cdot,0)$ is a non-zero element of the field $F$, the identity
$\Phi(z)=0$ can be rearranged as
\begin{equation*}
z^d=-c_d(\cdot,0)^{-1}\bigl(1+c_1(\cdot,0)\,z+\cdots+c_{d-1}(\cdot,0)\,z^{d-1}\bigr)\,,
\end{equation*}
so $z^d\in V$.
It follows that $zv\in V$ for all $v\in V$, and hence, by induction, that
$z^k\in V$ for all $k\ge0$, so $V$ is closed under multiplication.
Moreover, if $w\in V$ is non-zero then $v\mapsto wv$ defines an $F$-linear
map from $V$ to $V$ which is injective, since $\mathcal K$ is a field, and
therefore surjective, since $V$ is finite-dimensional; choosing $v\in V$
with $wv=1$ shows that $1/w\in V$.
The elements $\Psi_1(z)$ and $\Psi_2(z)$ belong to $V$, being polynomial
expressions in $z$ with coefficients in $F$, and $\Psi_1(z)\ne0$, so the
second identity in (\ref{eq:zy}) shows that $y=\Psi_2(z)/\Psi_1(z)$
belongs to $V$.
As $V$ is closed under multiplication, the $N+1$ elements $1,y,\ldots,y^N$
belong to $V$, so they are linearly dependent over $F$: there exist
$p_0,\ldots,p_N\in F$, not all zero, such that $\sum_{j=0}^Np_jy^j=0$.
Multiplying by the product of the denominators of the $p_j$ gives
$q_0,\ldots,q_N\in\e$, not all identically zero, such that
\begin{equation*}
\sum_{j=0}^Nq_j(t)\,y(t)^j=0 \quad\text{for all } t\in\R\,,
\end{equation*}
since $\sum_{j=0}^Nq_jy^j$ belongs to $\mathcal A$.
This is (\ref{eq:algrel}) with $D=N$, so Lemma~\ref{lem:bounded} implies
that $\lim_{t\to\infty}y(t)$ exists.
(In the language of field theory, $V=F(z)$, and $z$ and $y$ are algebraic
over $F$ of degree at most $N$; cf.~\cite[\S13.2, Prop.~12]{dummitfoote}.)
\end{proof}

The proof of the Theorem also yields
the following:

\begin{maincor}
Let $(X,\sigma)$ be a topologically mixing subshift of finite type, let
$f:X\to\R$ be locally constant, and for $t\in\R$ let $m_{tf}$ denote the
unique equilibrium measure for $tf$.
Then $\lim_{t\to\infty}m_{tf}(C)$ is an algebraic number for all
cylinders $C\subseteq X$.
Moreover, if $r\ge2$ is such that $f$ is constant on all cylinders of
length $r$, and $C\subseteq X$ is a cylinder of length at most $r$, then
the degree of $\lim_{t\to\infty}m_{tf}(C)$ is at most the number of
cylinders of length $r-1$.
\end{maincor}
\begin{proof}
Let $C\subseteq X$ be a cylinder, let $g$ be the indicator function of
$C$, and let $r\ge2$ be such that $f$ is constant on all cylinders of
length $r$ and $C$ has length at most $r$; such $r$ exists since $f$ is
locally constant.
Then $g$ is also constant on all cylinders of length $r$, so we may take
$B$, $N$, $\Delta$ and $c_n$ as in Section~\ref{sec:formula}, with $N$ the
number of cylinders of length $r-1$.
After passing to the higher block presentation, $g$ still takes values in
$\{0,1\}$, so the coefficients of the polynomials $\Phi$, $\Psi_1$ and
$\Psi_2$ in the proof of the Theorem belong to $\e_\Q$, by
(\ref{eq:coeffsQ}).
Repeating that proof with $\e_\Q$ in place of $\e$, and the field of
fractions of $\e_\Q$ in place of $F$, gives $q_0,\ldots,q_N\in\e_\Q$, not
all identically zero, satisfying (\ref{eq:algrel}) with $D=N$ for
$y(t)=m_{tf}(C)$.
The second part of Lemma~\ref{lem:bounded} then shows that
$\lim_{t\to\infty}m_{tf}(C)$ is an algebraic number of degree at most $N$.
\end{proof}

By Lemma~\ref{lem:extremal}, the limit measure $\lim_{t\to\infty}m_{tf}$
in the Corollary is a convex combination of ergodic
$f$-extremal measures, and the coefficients in this combination need not
be rational (see the examples in 
\cite[\S6.2]{cgu} and
\cite[\S6.1]{garibaldithieullen}).

\end{document}